\documentclass[11pt,reqno]{amsart}
\usepackage{amsmath}
\usepackage{amssymb}
\usepackage{amsthm}
\usepackage{enumerate}
\usepackage[mathscr]{eucal}
\usepackage{color}

\newcommand{\on}{\operatorname}

\newtheorem{fact}{Fact}[section]

\newtheorem{corollary}{Corollary}[section]
\newtheorem{remark}{Remark}[section]
\newtheorem{example}{Example}[section]
\newtheorem{proposition}{Proposition}[section]
\newtheorem{theorem}{Theorem}[section]

\theoremstyle{definition}
\newtheorem{df}{Definition}[section]
\author{Marek Balcerzak}
\address{Institute of Mathematics, Lodz University of Technology, al. Politechniki 8, 93-590
\L\'od\'z, Poland}
\email{marek.balcerzak@p.lodz.pl}

\author{\v Lubica Hol\'a}
\address{Academiy of Sciences, Institute of Mathematics, \v{S}tef\'anikova 49, 814 73 Bratislava, Slovakia}
\email{hola@mat.savba.sk}

\author{Du\v{s}an Hol\'y}
\address{Department of Mathematics and Computer Science, Faculty of Education, Trnava University, Priemyseln\'a, 918 43 Trnava, Slovakia}
\email{dusan.holy@truni.sk}

\title{Properties of equi-Baire~1 and equi-Lebesgue families of functions}
\date{}
\begin{document}

\begin{abstract}

We study several properties of equi-Baire~1 families of functions between metric spaces.
We consider the related equi-Lebesgue property for such families. We examine the behaviour
of equi-Baire~1 and equi-Lebesgue families with respect to pointwise and uniform convergence.
In particular, we obtain a criterion for a choice of a uniformly convergent subsequence from a sequence of functions that form an equi-Baire~1 family, which solves a problem posed in \cite{BKS}. Finally, we
discuss the notion of equi-cliquishness and relations between equi-Baire~1 families and sets of equi-continuity points.

\end{abstract}

\maketitle
\section{Introduction}
The notion of equi-Baire~1 families of functions introduced by Lecomte \cite{Le} and
rediscovered by Alikhani-Koopaei \cite{Al} was motivated  by the epsilon-delta
characterization of Baire~1 functions due to Lee, Tang and Zhao \cite{LTZ}.
Another old characterization of Baire~1 functions, due to Lebesgue \cite{Leb}, turns out useful in this context since, in some general cases, equi-Baire~1 families coincide with equi-Lebesgue ones.
The notion of equi-Lebesgue family is introduced in our paper, however it appeared without this name in \cite[Proposition 32]{Le} and \cite[Theorem~3.6]{Al}).

Equi-Baire~1 families of functions have been recently studied in \cite{BKS}. In our paper we continue these studies. We weaken some assumptions in several results of \cite{BKS}. We consider properties concerning pointwise and uniform convergence. In particular, we present a complete solution of the following Question~4.3 posed in \cite{BKS}:

\begin{quotation}
Given a uniformly bounded sequence $(f_n)$ of real-valued Baire~1 functions defined on a Polish space $X$, assume that the family $\{f_n:n \in \mathbb N\}$ is
equi-Baire~1. What more should be assumed to guarantee that $(f_n)$ contains a subsequence which is uniformly convergent on $X$?
\end{quotation}
Then we investigate other related properties.
We prove the following result: Let $(X,d_X)$ be a separable Baire metric space and $(Y,d_Y)$  be a  metric space. If  $\mathcal F$ is an equi-Baire~1  family of functions from $X$ to $Y$, then its set
$EC(\mathcal F)$ of equi-continuity points is dense $G_\delta$.
Also we consider the reverse situation where functions from a $G_{\delta}$ set $H\subset X$ to $Y$, that are equi-continuous on $H$, can be extended to functions from $X$ to $Y$ that form an
equi-Baire~1 family.

In what follows let $\mathbb N$ be the set of positive integers, $\mathbb R$ be the space of real numbers with the usual Euclidean metric and  $\mathbb R^+$  be the set of positive reals.

If $(Y,d)$ is a metric space, the open $d$-ball with center $y_0\in Y$ and radius $\varepsilon
>0$ will be denoted by $S(y_0,\varepsilon)$ and the closed $d$-ball with center $y_0\in Y$ and radius $\varepsilon>0$ will be denoted by $B(y_0,\varepsilon)$.  For a nonempty set $A\subset Y$, the symbol $S(A,\varepsilon)$ denotes the generalized open ball $\{y\colon \on{dist}(y,A)<\varepsilon\}$. The closure of $A\subset Y$ will be denoted by $\overline{A}$. Given $\varepsilon > 0$, we say that a set $A \subset Y$ is $\varepsilon$-discrete, if  $d(x,y) \ge \varepsilon$ for any distinct $x, y \in A$.
A topological space is called Polish if it is separable and completely metrizable.

The paper is organized as follows. In Section~1 we recall two characterizations of Baire~1 functions between metric spacess $X$ and $Y$, and establish minimal assumptions on $X$ and $Y$ under which these characterizations work. In Section~2 we define equi-Baire~1 and equi-Lebesgue families of functions and discuss relations between them. Sections~4 and 5 are devoted to results concerning
the behaviour of equi-Baire~1 and equi-Lebesgue families with respect to pointwise and uniform convergence, respectively. The final Section~6 contains other related results.
 We
discuss the notion of equi-cliquishness and relations between equi-Baire~1 families and sets of equi-continuity points.

\section{Two characterizations of Baire~1 functions}

A function $f$ between metric spaces $X$ and $Y$ is called Baire~1 if is $F_\sigma$-measurable, that is, whenever the preimage of any open set is $F_\sigma$ (cf. \cite[Def. 24.1]{Kech}).
Several characterizations of Baire~1 functions are known; some of them are due to Baire, see \cite[Thm 24.15]{Kech}. Usually some additional properties of $X$ and $Y$ are assumed. We will focus on two characterizations: one classical characterization is due to Lebesgue \cite{Leb}, and the other one was proposed quite recently in \cite{LTZ} and \cite{FC}.

The Lebesgue characterization was originally formulated for real-valued Borel measurable functions of arbitrary class $\alpha >0$. A more general version can be found in \cite{Ku}. Below, we present it
for $\alpha=1$, in a bit modified form, and we give its proof for a reader's convenience.

\begin{theorem} \label{HL1} \rm{(cf. \cite[p. 375]{Ku})}
 Let $(X,d_X)$ and $(Y,d_Y)$ be  metric spaces. For $f\colon X \to Y$
 consider the following conditions:
\begin{itemize}
\item[(1)] A function $f$ is Baire~1;
\item[(2)] For each $\varepsilon >0$ there is a cover $(X_i)_{i \in \mathbb N}$  of $X$ consisting of closed sets such that $\on{diam}f(X_i) \le\varepsilon$, for all $i \in \mathbb N$.
\end{itemize}
Then \rm{(2) $\Rightarrow$ (1)} holds. If $(Y,d_Y)$ is separable, then
\rm{(1) $\Rightarrow$ (2)} is true.

\end{theorem}
\begin{proof} (1) $\Rightarrow$ (2): Let $f$ be Baire~1. Since $(Y,d_Y)$ is separable,
for a fixed $\varepsilon >0$ we can find open balls $B_1,B_2, \dots$ in $Y$ such that
$Y=\bigcup_{n\in\mathbb N} B_n$ and $\on{diam}B_n\le\varepsilon$ for all $n\in\mathbb N$.
For $n\in\mathbb N$, the set $f^{-1}(B_n)$ is of type $F_\sigma$, so it can be expressed as the union $\bigcup_{j\in\mathbb N}F_{n,j}$ of closed sets $F_{n,j}$, $j\in\mathbb N$.
Enumerate $\{F_{n,j}: n,j\in\mathbb N\}$ as $\{X_i: i\in\mathbb N\}$ and observe that the respective condition in (2) holds.

(2) $\Rightarrow$ (1): By (2), for each $k\in\mathbb N$ we find closed sets $X_{n,k}$, $n\in\mathbb N$
such that $X=\bigcup_{n\in\mathbb N}X_{k,n}$ and $\on{diam} f(X_{k,n})<1/k$ for all $n\in\mathbb N$.
Let $U$ be an open set in $Y$. Define $\mathcal F=\{X_{k,n}: f(X_{k,n})\subset U\}$.
It is enough to show that $f^{-1}(U)=\bigcup\mathcal F$ since then $f^{-1}(U)$ is of type $F_\sigma$.

Firstly, if $x\in \bigcup\mathcal F$, then $x\in X_{k,n}$  for some $k,n\in\mathbb N$ with $f(X_{k,n})\subset U$. Hence $f(x)\in U$ and so, $x\in f^{-1}(U)$.
Secondly, if $f(x)\in U$, we can find $k\in\mathbb N$ such that every $y\in Y$ with $d_Y(f(x),y)<1/k$
is in $U$. Find $n\in\mathbb N$ such that $x\in X_{k,n}$. Since $\on{diam} f(X_{k,n})<1/k$, we have $f(X_{k,n})\subset U$. Hence $x\in\mathcal F$ as desired.
\end{proof}

If a function $f$ between metric spaces $X$ and $Y$ satisfies condition (2) from Theorem \ref{HL1},
we will say that $f$ has {\em the Lebesgue property}.

Thus, by implication (2)~$\Rightarrow$~(1) in Theorem \ref{HL1}, if $f\colon X\to Y$ has the Lebesgue property, then it is Baire~1.
We observe that the separability of $Y$ is necessary for implication (1)~$\Rightarrow$~(2)
in Theorem \ref{HL1}.

\begin{proposition} \label{HLo}
Let $(Y,d_Y)$ be a metric space. The following are equivalent:
\begin{itemize}
\item[(i)] $(Y,d_Y)$ is separable;
\item[(ii)] For each metric space $(X,d_X)$, every Baire~1 function $f\colon X \to Y$ has the Lebesgue property.
\end{itemize}
\end{proposition}

\begin{proof} We only need to show (ii) $\Rightarrow$ (i).  Suppose that $(Y,d_Y)$ is not separable. Then there is $\varepsilon > 0$ and an uncountable $\varepsilon$-discrete set $L \subset Y$.  Put $L = \{y_i: i \in I\}$. Then $\{B(y_i,\varepsilon/4): i \in I\}$ is an uncountable family of pairwise disjoint closed balls.

Let $X = Y$, $d_X = d_Y$ and $f\colon X \to Y$ be the identity function. Then $f$ is Baire~1. By (ii)  there is a cover $(X_n)_{n \in \mathbb N}$  of $X$ consisting of closed sets such that $\on{diam}f(X_n) \le\varepsilon/4$, for all $n \in \mathbb N$. There must exist $m \in \mathbb N$ such that $X_{m} \cap B(y_i,\varepsilon/4) \ne \emptyset$ for uncountable many balls $B(y_i,\varepsilon/4)$, a contradiction.
\end{proof}

Another characterization of Baire~1 functions, formulated in epsilon-delta style,
is due to Lee, Thang and Zhao \cite{LTZ}. Their result concerns the case of functions between Polish spaces. The following generalization was obtained by Fenecios and  Cabral in \cite{FC}.

\begin{theorem}\cite{FC} \label{FaC}
Let $(X,d_X)$  and $(Y,d_Y)$ be  separable metric spaces. The following are equivalent:
\begin{itemize}
\item[(1)] A function $f\colon X \to Y$ is Baire~1;
\item[(2)] For each $\varepsilon > 0$ there is a  function $\delta: X \to \mathbb R^+$ such that for any $x, y \in X$, if $d_X(x,y) < \min\{\delta(x),\delta(y)\}$, then $d_Y(f(x),f(y)) < \varepsilon$.
\end{itemize}
\end{theorem}

If a function $f$ between metric spaces $X$ and $Y$ satisfies condition (2) from Theorem \ref{FaC},
we will say that $f$ has {\em LTZ-property}. What is interesting, the proof of Theorem ~5 in \cite{FC} uses the Lebesgue property and a careful analysis of this proof shows that the following implications are valid.

\begin{proposition} \label{Novi}
Let $(X,d_X)$, $(Y,d_Y)$ be  metric spaces and $f\colon X\to Y$.
\begin{itemize}
\item[(i)] If $X$ is separable and $f$ has LTZ-property, then $f$ has the Lebesgue property.
\item[(ii)] If $f$ has the Lebesgue property, then $f$ has LTZ-property.
\end{itemize}
\end{proposition}

\begin{remark}
Using Theorem \ref{HL1} and Proposition \ref{Novi}, we obtain Theorem \ref{FaC} where
the separability of $Y$ can be omitted. Also, the separability of $X$ is needed only for
implication \rm{(2) $\Rightarrow$ (1)} in Theorem \ref{FaC}.
\end{remark}

We observe that the separability of $X$ is necessary for implication in Proposition \ref{Novi} (i).

\begin{proposition} \label{Hol11}
Let $(X,d_X)$ be a metric space. The following are equivalent:
\begin{itemize}
\item[(1)] $(X,d_X)$ is separable;
\item[(2)] For each metric space $(Y,d_Y)$, every function $f\colon X \to Y$ with LTZ-property has the Lebesgue property.
\end{itemize}
\end{proposition}
\begin{proof} We only need to show (2) $\Rightarrow$ (1).  We mimic the proof of Proposition \ref{HLo}. Suppose that $(X,d_X)$ is not separable. Then there is $\varepsilon > 0$ and an uncountable
$\varepsilon$-discrete set $L \subset X$. Put $L = \{x_i: i \in I\}$.
Then $\{B(x_i,\varepsilon/4): i \in I\}$ is an uncountable family of pairwise disjoint closed balls.

Put $Y = X$, $d_Y = d_X$ and let $f\colon X \to Y$ be the identity function which surely has LTZ-property.
Thus by (2) it has the Lebesgue property. However, for $\varepsilon/4$  there is no cover $(X_i)_{i \in \mathbb N}$  of $X$ consisting of closed sets such that $\on{diam}f(X_i) \le \varepsilon/4$ for all $i \in \mathbb N$, a contradiction.
\end{proof}

\begin{remark} \label{addi}
Let $(X,d_X)$, $(Y,d_Y)$ be  metric spaces and $(X, d_X)$ is hereditary Baire. By \cite[Thm C]{KM},
a function $f\colon X\to Y$  is Baire~1 if and only if it has LTZ-property.
\end{remark}

\section{Equi-Baire 1 and equi-Lebesgue families}

The characterization of Baire~1 functions by LTZ-property motivated Lecomte in his paper \cite{Le} to introduce the notion of equi-Baire~1 family of functions. Then it was rediscovered by Alighani-Koopaei \cite{Al} who showed its applications in dynamical systems; see also his recent paper \cite{A2}.

\begin{df}
Let $(X,d_X)$ and $(Y,d_Y)$  be metric spaces. We say that a family $\mathcal F$ of functions from $X$ to $Y$ is equi-Baire~1 if, for each $\varepsilon > 0$, there exists a function $\delta\colon X \to \mathbb R^+$ such that, for all $f \in \mathcal F$  and $x, y \in X$ the condition $d_X(x,y) < \min\{\delta(x),\delta(y)\}$ implies
$d_Y(f(x),f(y)) \le \varepsilon$.
\end{df}

\begin{remark} \label{HLL}
Let $(X,d_X)$ and $(Y,d_Y)$ be  metric spaces.
\begin{itemize}
\item[(i)] By Proposition  \ref{Novi}(i), if $(X,d_X)$ is separable, then every member of an
equi-Baire~1 family of functions from $X$ to $Y$ is a Baire~1 function.
\item[(ii)] By Remark \ref{addi}, if $(X,d_X)$ is hereditarily Baire, then every member of an
equi-Baire~1 family of functions from $X$ to $Y$ is a Baire~1 function.
\end{itemize}
\end{remark}

We suggest the following definition of an equi-Lebesgue family.

\begin{df}
Let $(X,d_X)$ and $(Y,d_Y)$  be metric spaces. We say that a family $\mathcal F$ of functions from $X$ to $Y$ is equi-Lebesgue if,
for every $\varepsilon > 0$, there is a cover $(X_i)_{i \in \mathbb N}$  of $X$ consisting of closed sets such that $\on{diam}f(X_i) \le \varepsilon$ for all $i \in \mathbb N$ and $f \in \mathcal F$.
\end{df}

An advantage of the notion of an equi-Lebesgue family is that every member of an equi-Lebesgue family of functions from $X$ to $Y$ is a Baire~1 function (see Theorem \ref{HL1}).

From the proof of Theorem 3.6 in \cite{Al} we can deduce a general fact where the assumptions on spaces $X$ and $Y$ are weakened.

\begin{fact} \cite[Theorem 3.6]{Al}\label{Ali}
Let $(X,d_X)$ and $(Y,d_Y)$  be metric spaces and $\mathcal F$ be a family of functions from $X$ to $Y$.
\begin{itemize}
\item [(i)] If $\mathcal F$ is equi-Lebesgue, then it is equi-Baire~1.
\item [(ii)] Assume that $(X,d_X)$ is separable. If $\mathcal F$ is equi-Baire~1, then $\mathcal F$
is equi-Lebesgue.
\end{itemize}
\end{fact}

The separability of $X$ in Fact \ref{Ali} (ii) is necessary. Indeed, we use the idea from the proof of
Proposition~\ref{Hol11}. Assume that $(X,d_X)$ is not separable and take the identity function $f\colon X\to X$.
We know that $f$  is a function with the  LTZ-property and  without the Lebesgue property. Thus $\mathcal F=\{ f\}$ is equi-Baire~1 and it is not equi-Lebesgue.

\section{Properties concerning pointwise convergence}
Let $X$ be a topological space, $(Y,d)$ be a metric space and $Y^X$ be the space of all functions from $X$ to $Y$.
We use some notions of general topology, for instance, uniformities and nets. For more information on them, we refer the reader to \cite{Ke}.

Recall that the topology $\tau_p$ of pointwise
convergence on $Y^X$ is induced by the uniformity $\mathfrak U_p$ of
pointwise convergence which has a base consisting of sets of the
form
$$W(A,\varepsilon )=\{(f,g):\ \forall\ x\in A\ \ d(f(x),g(x))<
\varepsilon \},$$
where $A$ is a finite subset of $X$ and $\varepsilon >0$. The general
$\tau_p$-basic neighborhood of $f\in Y^X$ will be denoted by
$W(f,A,\varepsilon )$, i.e. $W(f,A,\varepsilon
)=W(A,\varepsilon )[f]$. In other words,
$$W(f,A,\varepsilon )=\{g \in Y^X:\ d(f(x),g(x)) <\varepsilon\ \text{for every}\ x \in A\}.$$

The following result is classical.

\begin{theorem} \cite[Chapter 7, Theorem 14]{Ke}
Let $X$ be a topological space and $(Y,d)$ be a metric space. If a family $\mathcal F$ of functions from $X$ to $(Y,d)$ is equicontinuous at $x \in X$, then the closure of $\mathcal F$ relative to the topology of pointwise convergence is also equicontinuous at $x$.
\end{theorem}

Thus if a net $(f_\sigma)$  of continuous functions from a topological space $X$ to a metric space $(Y,d)$ pointwise converges to a function $f\colon X \to Y$ and the net $(f_\sigma)$ is equicontinuous, then $f$ is continuous.  We will find an analogy for Baire~1 functions.

The following proposition is easy to verify.

\begin{proposition} \label{lu}
Let $(X,d_X)$ and $(Y,d_Y)$ be metric spaces. Let $\mathcal F$ be an equi-Lebesgue family of functions from $X$ to $Y$.
The closure of $\mathcal F$ relative to the topology of pointwise convergence is equi-Lebesgue too.
\end{proposition}

We have the following corollaries.

\begin{corollary} \label{lh1}
Let $(X, d_X)$, $(Y, d_Y)$ be metric spaces and   $(f_\sigma)$ be a net of Baire~1 functions from $X$ to $Y$ which pointwise converges to a function $f\colon X \to Y$. If  the net $(f_\sigma)$ is equi-Lebesgue,  then $f$ is Baire~1.
\end{corollary}

\begin{proof}
Use Proposition \ref{lu} and Theorem \ref{HL1}.
\end{proof}

\begin{corollary} \label{lhh}
Let  $(X, d_X)$, $(Y, d_Y)$ be metric spaces and $X$ be separable. Let $(f_\sigma)$ be a net of Baire~1 functions from $X$ to $Y$ which pointwise converges to a function $f\colon X \to Y$. If  the net $(f_\sigma)$ is equi-Baire~1,  then $f$ is Baire~1.
\end{corollary}

\begin{proof}
Use Fact \ref{Ali} (ii) and Corollary \ref{lh1}.
\end{proof}

As Example~4.1  in \cite{BKS} shows, the condition of equi-Lebesgue (equi-Baire ~1) is not necessary to guarantee that the pointwise limit of a sequence of Baire~1 functions is Baire~1. This example
uses $X=Y=\mathbb R$ with the Euclidean metric. Let $\{q_n: n \in \mathbb N\}$ be an enumeration of rational numbers. For every $n \in \mathbb N$ put $f_n = \chi_{\{q_n\}}$. The sequence $(f_n)$ is pointwise convergent to the zero function. However, $\{f_n: n \in \mathbb N\}$ is not equi-Baire~1,
and it is not equi-Lebesgue by Fact~\ref{Ali} (i).

This example can be generalized as follows.

\begin{proposition} \label{HL7} Let $(X,d)$ be a metric space. Consider a non-empty family $\mathcal A$ of pairwise disjoint subsets of $(X,d)$. Then the family $\{\chi_A: A \in \mathcal A\}$ is equi-Lebesgue  if and only if $\mathcal A$ is countable, each element of $\mathcal A$ is $F_\sigma$ and $\bigcup \mathcal A$ is $G_\delta$.
\end{proposition}

\begin{proof}
Suppose that $\mathcal A$ is a family of pairwise disjoint subsets of $(X,d)$ such that  $\mathcal A$ is countable, each element of $\mathcal A$ is $F_\sigma$ and $\bigcup \mathcal A$ is $G_\delta$. Then it is easy to verify that the family $\{\chi_A: A \in \mathcal A\}$ is equi-Lebesgue.

Suppose now that $\mathcal A$ is a family of pairwise disjoint subsets of $(X,d)$ such that the family $\{\chi_A: A \in \mathcal A\}$ is equi-Lebesgue. Since for every $A \in \mathcal A$ the function $\chi_A$ must be Baire~1, every $A \in \mathcal A$ has to be $F_\sigma$.  By the assumption  there is a countable cover $\mathcal C$  of $X$ consisting of closed sets such that $\operatorname{diam}\chi_A(C) < 1/2$, for all $C \in \mathcal C$ and  $A \in \mathcal A$. Suppose that $\mathcal A$ is uncountable. There must exist $C \in \mathcal C$ such that $C \cap A \ne \emptyset$ for uncountable many $A \in \mathcal A$, a contradiction.

Let us show that $\bigcup \mathcal A$ is $G_\delta$. Since $\mathcal C$ is a cover of $X$, for every $A \in \mathcal A$ there is $C \in \mathcal C$ such that $A \cap C \ne \emptyset$, and this implies that
$C \subset A$ because $\on{diam} \chi_A(C)<1/2$ means that $\chi_A(C)=\{1\}$.  Hence for every $A \in \mathcal A$ we have $A = \bigcup\{C\in\mathcal C: C \subset A\}$. Thus $X \setminus \bigcup \mathcal A$ has to be an $F_\sigma$ set.
\end{proof}

Note that Proposition \ref{HL7} recovers Corollary 3.17 in \cite{BKS} which however works when $X$ is a
separable hereditary Baire metric space
since its proof is a conequence of Theorem 3.16 in \cite{BKS}. Here we deal with a more general case and we give a direct proof.

\begin{proposition} \label{h15}
Let $(X,d_X)$, $(Y,d_Y)$ be metric spaces and let $(f_n)$ be a sequence of functions $f_n:X \to Y$,
$n \in \mathbb N$, that form an equi-Lebesgue family. If $\overline{\{f_n(x): n \in \mathbb N\}}$ is compact in $(Y,d_Y)$ for every $x \in X$,
then there exists a subsequence $(f_{n_k})$ which is pointwise convergent to a Baire~1 function.
\end{proposition}

\begin{proof}
For every $x \in X$ put $K_x = \overline{\{f_n(x): n \in \mathbb N\}}$.
The set $\Pi_{x \in X} K_x$ is compact and $f_n \in \Pi_{x \in X} K_x$ for every $n \in \mathbb N$. There is a function $f \in \Pi_{x \in X} K_x$ such that $f$ is a cluster point of $(f_n)$ in $\Pi_{x \in X} K_x$. By Proposition \ref{lu}, $\mathcal F = \{f_n: n \in \mathbb N\} \cup \{f\}$ is equi-Lebesgue and thus $f$ is Baire~1.

Since $\mathcal F$ is equi-Lebesgue, for every $k \in \mathbb N$ there is a cover $(X_{i}^k)_{i \in \mathbb N}$  of $X$ consisting of closed sets such that $\on{diam}g(X_{i}^k) \le 1/k$ for all $i \in \mathbb N$ and $g \in \mathcal F$. There is a countable set $H \subset X$ such that $H \cap X_{i}^k \ne \emptyset$ for all $k, i \in \mathbb N$.

The set $\Pi_{x \in H} K_x$ is first countable (even metrizable) and  $f|H$ is a cluster point of $(f_n|H)$ in $\Pi_{x \in H} K_x$. Thus there is a subsequence $(f_{n_k}|H)$ which converges to $f|H$ in $\Pi_{x \in H} K_x$. We will show that the sequence $(f_{n_k})$ pointwise converges to $f$.

Let $x \in X$ and $\varepsilon > 0$. Let $k \in \mathbb N$ be such that $3/k < \varepsilon$. There is a cover $(X_{i}^k)_{i \in \mathbb N}$  of $X$ consisting of closed sets such that $\on{diam}g(X_{i}^k) \le 1/k$ for all $i \in \mathbb N$ and $g \in \mathcal F$. Let $i \in \mathbb N$ be such that $x \in X_{i}^k$. Let $z \in X_{i}^k \cap H$. There is $l_0 \in \mathbb N$ such that $d_Y(f_{n_l}(z),f(z)) < 1/k$ for every $l \ge l_0$. Let $l \ge l_0$. Then
$$d_Y(f_{n_l}(x),f(x)) \le d_Y(f_{n_l}(x),f_{n_l}(z)) + d_Y(f_{n_l}(z)),f(z)) + d_Y(f(z),f(x))
< 3/k < \varepsilon.$$
\end{proof}

We have the following generalization of Proposition 4.5 in \cite{BKS}.

\begin{corollary} \em{(\cite{BKS} for $Y = \mathbb R$)}
Let $(X,d_X)$ be a separable metric space,  $(Y,d_Y)$ be a metric space and let $(f_n)$ be a sequence of functions $f_n:X \to Y$, $n \in \mathbb N$, that form an equi-Baire~1 family. If $\overline{\{f_n(x): n \in \mathbb N\}}$ is compact in $(Y,d_Y)$ for every $x \in X$, then there exists a subsequence $(f_{n_k})$ which is pointwise convergent to a Baire~1 function.
\end{corollary}
\begin{proof}
From Fact \ref{Ali}  we know that if $(X,d_X)$ is separable then the notions of
equi-Baire~1 family and equi-Lebesgue family coincide. Thus we can use Proposition \ref{h15}.
\end{proof}

The following result was obtained as Theorem 3.3 in \cite{BKS} where $X$ and $Y$ are Polish spaces.
By Fact \ref{Ali} we observe that the same proof works if $X$ is only a metric space and $Y$ is a separable metric space.

\begin{proposition} Let $(X,d_X)$ be a metric space and $(Y,d_Y)$ be a separable metric space.
 Let $\mathcal F = \{f_n: n \in \mathbb N\}$ be a family of continuous functions from $X$ to $Y$, such that the sequence $(f_n)$ is pointwise convergent on $X$. Then $\mathcal F$ is an equi-Lebesgue family.
\end{proposition}

%Using Remark \ref{hol11}  we see that  Theorem 3.3 in \cite{BKS} holds if $(X,d_X)$ is only a  metric space and  $(Y,d_Y)$ is a separable metric space.

\section{Properties concerning uniform convergence}

Let us define the topology $\tau_{U}$ of uniform convergence  on $Y^X$. This topology is induced by the
uniformity $\mathfrak U_{U}$  which has a base consisting of sets of the
form
$$W(\varepsilon )=\{(f,g):\ \forall\ x\in X\ \ d(f(x),g(x))<
\varepsilon \},$$
where $\varepsilon >0$. The general
$\tau_{U}$-basic neighborhood of $f\in Y^X$ will be denoted by
$W(f,\varepsilon )$, i.e. $W(f,\varepsilon)=W(\varepsilon )[f]$. In other words,
$$W(f,\varepsilon) = \{g \in Y^X:\ d(f(x),g(x)) < \varepsilon\ \text{for every}\ x \in X\}.$$
It is easy to verify that the countable family $\{W(1/n): n \in \mathbb{N}\}$  is a base of the
uniformity $\mathfrak U_{U}$.  Thus the uniformity $\mathfrak U_{U}$  is metrizable \cite{Ke}.

Finally, let $(X,d)$ be a metric space. We will define the topology $\tau_{UB}$ of uniform convergence on
bounded sets on $Y^X$. This topology is induced by the
uniformity $\mathfrak U_{UB}$  which has a base consisting of sets of the
form
$$W(B,\varepsilon )=\{(f,g):\ \forall\ x\in B\ \ d(f(x),g(x))<
\varepsilon \},$$
where $B$ is a bounded subset of $X$ and $\varepsilon >0$. A general
$\tau_{UB}$-basic neighborhood of $f\in Y^X$ will be denoted by
$W(f,B,\varepsilon )$, i.e. $W(f,B,\varepsilon)=W(B,\varepsilon )[f]$. In other words,
$$W(f,B,\varepsilon) = \{g \in Y^X:\ d(f(x),g(x)) < \varepsilon\ \text{for every}\ x \in B\}.$$
It is easy to verify that the uniformity  $\mathfrak U_{UB}$ is metrizable. Let $x_0$ be an arbitrary point in $X$. For every $n \in \mathbb{N}$ put
$B_n = B(x_0,n)$. Then the countable family $\{W(B_n,1/k): n, k \in  \mathbb{N}\}$  is a base of the uniformity $\mathfrak U_{UB}$.

\begin{proposition}
Let  $(f_n)$ be a sequence of bounded  real-valued Baire~1 functions defined on a metric  space $X$. If $(f_n)$ is uniformly convergent to $f\colon X \to \mathbb{R}$, then for every $\varepsilon > 0$ there is a finite cover $\mathcal B$ of $X$ consisting  of subsets of $X$ which are  $F_\sigma$ and  such that $\on{diam}f_{n}(B) \le \varepsilon$ for all $B \in \mathcal B$ and $n \in \mathbb N$.
\end{proposition}
\begin{proof} Let $\varepsilon > 0$.
Put $\mathcal F = \{f, f_1, f_2, \dots, f_n, \dots \}$. Then $\mathcal F$ is a compact set in $(\mathbb{R}^X,\tau_U)$. There are functions $f_{k_1}, f_{k_2}, \dots f_{k_n} \in \mathcal F$ such that
	$$\mathcal F \subseteq \bigcup \{W(f_{k_j},\varepsilon/3): j\in\{1,\dots, n\}\}.$$
For simplicity  put $g_j = f_{k_j}$ for every $j\in\{1,\dots, n\}.$  Let $M > 0$ be such that
$$\bigcup \{g_j(X): j\in\{1,\dots, n\}\} \subset [-M,M].$$
Let $\mathcal V = \{V_1, \dots, V_m\}$ be a finite open cover of $[-M,M]$, where the diameter of members of this cover is less than $\varepsilon/3$. For each $i\in\{1,\dots, m\}$ and $j\in\{1,\dots, n\}$ let
$B_i^j  = g_j^{-1}(V_i)$.
Since $g_j$ is a Baire~1 function, the set $B_i^j$ is $F_\sigma$ for all $j\in\{1,\dots, n\}$ and $i\in\{1,\dots, m\}$. Put
$$\mathcal B' = \{B_{i_1}^1 \cap \dots \cap B_{i_n}^n: i_j \in \{1,\dots, m\}\ \text{for each}\ j \in \{1,\dots, n\}\}.$$
Every set $B \in \mathcal B'$ is of type $F_\sigma$. Let $\mathcal{B}$ be the family consisting of all nonempty sets from $\mathcal B'$.  Clearly $X =\bigcup\mathcal B$. One can check that $|f_i(p) - f_i(q)| <\varepsilon$
for all $f_i \in \mathcal F$, $B\in\mathcal{B}$ and $p, q \in B$.
\end{proof}

\begin{theorem} \label{HH}
Let  $(f_n)$ be a sequence of bounded  real-valued Baire~1 functions defined on a Baire metric  space $X$. If $(f_n)$ is uniformly convergent to $f\colon X \to \mathbb{R}$, then for every $\varepsilon > 0$ there is a finite cover $\mathcal B$ of $X$ consisting  of subsets of $X$ which are either open or  $F_\sigma$ nowhere dense and  such that $\on{diam}f_{n}(B) \le \varepsilon$ for all $B \in \mathcal B$ and $n \in \mathbb N$.
\end{theorem}

\begin{proof}
Let $\varepsilon > 0$.
Put $\mathcal F = \{f, f_1, f_2, \dots, f_n, \dots \}$. Then $\mathcal F$ is a compact set in $(\mathbb{R}^X,\tau_U)$. There are functions $f_{k_1}, f_{k_2}, \dots , f_{k_n} \in \mathcal F$ such that
	$$\mathcal F \subseteq \bigcup \{W(f_{k_j},\varepsilon/3): j\in\{1,\dots, n\}\}.$$
For simplicity  put $g_j = f_{k_j}$ for every $j\in\{1,\dots, n\}\}.$  Let $M > 0$ be such that
$$\bigcup \{g_j(X): j\in\{1,\dots, n\}\} \subset [-M,M].$$
Let $\mathcal V = \{V_1, \dots, V_m\}$ be a finite open cover of $[-M,M]$, where the diameter of members of this cover is less than $\varepsilon/3$.
	For each $i\in\{1,\dots, m\}$ and $j\in\{1,\dots, n\}$ let
	$$L_i^j  = \{z \in C(g_j): g_j(z) \in V_i\},$$
	where $C(g_j) = \{z \in X:\ g_j\ \text{is continuous at}\ z\}$. Since $X$ is a Baire space, $C(g_j)$ is a dense $G_\delta$ set in $X$ for every $j\in\{1,\dots, n\}$.

 For each $z \in L_i^j$, let $W_i^j(z)$ be an open neighbourhood of $z$ such that $g_j(W_i^j(z)) \subset V_i$.  Let
	$$W_i^j = \bigcup \{W_i^j(z):\ z \in L_i^j\}$$	
and	
	$$\mathcal U = \{W_{i_1}^1 \cap \dots \cap W_{i_n}^n: i_j \in \{1,\dots, m\}\ \text{for each}\ j \in \{1,\dots, n\}\}.$$
For each $U \in \mathcal U$, let
	$$P_i^j(U) = \{z \in \overline U \setminus U:\ g_j(z) \in V_i\}$$
and		
	$$\mathcal P = \{P_{i_1}^1(U) \cap \dots \cap P_{i_n}^n(U): U \in \mathcal U \text{ and }i_j \in \{1,\dots, m\}\ \text{for each}\ j \in \{1,\dots, n\}\}.$$
Since $g_j$ is a Baire~1 function, the set $P_i^j(U)$ is of type $F_\sigma$ for all $j\in\{1,\dots, n\}$, $i\in\{1,\dots, m\}$
and $U \in \mathcal U$. Thus also every set $P \in \mathcal P$ is of type $F_\sigma$ and of course nowhere dense.

We show that $X = \bigcup\mathcal U\cup \bigcup\mathcal P$. The set $\bigcup\mathcal U$ is open and dense. Let $G$ be a nonempty open set in $X$. Choose some $u\in G\cap C(g_1)\cap \dots\cap C(g_n)$. Then for each $j\in \{1,\dots, n\}$ we have $u\in W_{i_j}^j$ for some $i_j\in\{1,\dots, m\}$. Thus  $u\in G\cap W_{i_1}^1 \cap \dots \cap W_{i_n}^n$, i.e. $G \cap \bigcup\mathcal U \ne \emptyset.$

Let $z\in X \setminus \bigcup\mathcal U$. Then $z\in\overline{\bigcup\mathcal U}$. Hence there is $U\in\mathcal U$ such that $z\in\overline U\setminus U$. It follows that for every $j \in \{1,\dots, n\}$ there is  $i_j \in \{1,\dots, m\}$ such that $z\in P_{i_j}^j(U)$ and then $z\in\bigcup \mathcal P$. Hence $X = \bigcup\mathcal U\cup \bigcup\mathcal P$.
	
Finally, let $\mathcal{B}$ be the family consisting of all nonempty sets from $\mathcal U$ and  $\mathcal P$.
One can check that $|f_i(p) - f_i(q)| <\varepsilon$ for all $f_i \in \mathcal F$, $B\in\mathcal{B}$ and $p, q \in B$.
\end{proof}

\begin{theorem} \label{lh}
Let  $(f_n)$ be a pointwise bounded sequence of real-valued  functions defined on a topological space $X$. Suppose that for every $\varepsilon > 0$ there is a finite cover $\mathcal B$ of $X$ such that $\on{diam}f_{n}(B) \le \varepsilon$ for every $B \in \mathcal B$ and for every $n \in \mathbb N$. Then $(f_n)$ contains a subsequence which is uniformly convergent on $X$.
\end{theorem}

\begin{proof}
For every $x \in X$ let $M_x > 0$ be such that
$$|f_n(x)| \le M_x\;\mbox{ for every  }\;n \in \mathbb N.$$
%As before denote by $B(X,[-M,M])$ the family of all functions from $X$ to $[-M,M]$. Let $d$ be the usual metric of uniform convergence on $B(X,[-M,M])$. If $f \in B(X,[-M,M])$ %and $r \in \Bbb R$, $r > 0$, put $S(f,r) = \{h: h \in B(X,[-M,M]), d(f,h) < r\}$.
The set $\Pi_{x \in X}[-M_x,M_x]$ is compact and for every $n \in \mathbb N$,  $f_n \in \Pi_{x \in X}[-M_x,M_x]$. There is a function $f \in \Pi_{x \in X}[-M_x,M_x]$ such that $f$ is a cluster point of $(f_n)$ in $\Pi_{x \in X}[-M_x,M_x]$. There is a subnet $(g_\sigma)$ of $(f_n)$ which pointwise converges to $f$.
First, observe the following: Let  $\eta > 0$  and $\mathcal B$ be a finite cover of $X$ such that $\on{diam}g_\sigma(B) \le \eta$ for all $B \in \mathcal B$ and every $\sigma$. Then also $\on{diam}f(B) \le \eta$ for every $B \in \mathcal B$.

We will show that $(g_\sigma)$ uniformly converges to $f$. Let $\varepsilon > 0$. There is a finite cover $\mathcal D$ of $X$ such that $\on{diam}g_\sigma(D) \le \varepsilon/3$ for every $D \in \mathcal D$ and for every $\sigma$ and also $\on{diam}f(D) \le \varepsilon/3$ for every $D \in \mathcal D$. For each $D \in \mathcal D$ choose $x_D \in D$. For every $D \in \mathcal D$ there is $\sigma_D$ such that $|f(x_D) - g_\sigma(x_D)| < \varepsilon/3$ for every $\sigma \ge \sigma_D$. Let $\sigma_0$ be such that $\sigma_0 \ge \sigma_D$ for all $D \in \mathcal D$. Such $\sigma_0$ exists since the cover $\mathcal D$ is finite. We claim that
$$|g_\sigma(x) - f(x)| < \varepsilon\;\mbox{ for every }\; \sigma \ge \sigma_0\;\mbox{  and for every }\;x \in X.$$
Let $x \in X$. There is $D \in \mathcal D$ such that $x \in D$. Let $\sigma \ge \sigma_0 \ge \sigma_D$. Then
$$|g_\sigma(x) - f(x)| \le |g_\sigma(x) - g_\sigma(x_D)| + |g_\sigma(x_D) - f(x_D)| + |f(x_D) - f(x)| $$
$$<\varepsilon/3 + \varepsilon/3 + \varepsilon/3 = \varepsilon.$$
Thus $(g_\sigma)$ uniformly converges to $f$, i.e. $(g_\sigma)$ converges to $f$  in the metrizable space $(\mathbb{R}^X,\tau_U)$.  Then $f$ is a cluster point of $(f_n)$  in $(\mathbb{R}^X,\tau_U)$. Since $(\mathbb{R}^X,\tau_U)$ is metrizable, there is a subsequence $(f_{n_k})$ of $(f_n)$ which converges to $f$ in $(\mathbb{R}^X,\tau_U)$; i.e.
$(f_{n_k})$ uniformly converges to $f$.
\end{proof}

Notice that Theorem \ref{lh} works also for functions from a topological space $X$ to a metric space $(Y,d)$ under the assumption that $\overline{\{f_n(x): n \in \mathbb N\}}$ is compact for every $x \in X$.

%\begin{corollary} \label{C2}
%Let  $(f_n)$ be a uniformly bounded sequence of real-valued  functions defined on a topological space $X$. Suppose that for every $\varepsilon > 0$ there is a finite cover $\mathcal B$ of $X$ such that $\on{diam}f_{n}(B) \le \varepsilon$ for every $B \in \mathcal B$ and for every $n \in \mathbb N$. Then $(f_n)$ contains a subsequence which is uniformly convergent on $X$.
%\end{corollary}

We have the following answer to Question 4.3 from \cite{BKS}.

\begin{theorem}
Let $(f_n)$ be a uniformly bounded sequence  of real-valued Baire~1 functions defined on a Baire metric space $X$. The following are equivalent:
\begin{itemize}
\item[(1)] $(f_n)$ contains a subsequence which is uniformly convergent on $X$;
\item[(2)] There is a subsequence $(f_{n_k})$ of $(f_n)$ with the following property: for every $\varepsilon > 0$ there is a finite cover $\mathcal B$ of $X$ consisting  of subsets of $X$ which are either open or $F_\sigma$ nowhere dense such that $\on{diam}f_{n_k}(B) \le \varepsilon$ for every $B \in \mathcal B$ and for every $k \in \mathbb N$.
\end{itemize}
\end{theorem}

\begin{proof}
$(1) \Rightarrow (2)$  follows from Theorem \ref{HH}.

$(2) \Rightarrow (1)$: By $(2)$ consider a subsequence $(f_{n_k})$ of $(f_n)$ described in (2). By Theorem \ref{lh} there is a subsequence of the sequence $(f_{n_k})$ which uniformly converges on $X$. However, this subsequence of $(f_{n_k})$ is also a subsequence of $(f_n)$. Thus $(f_n)$ contains a subsequence which is uniformly convergent on $X$.
\end{proof}

We propose the following generalization of Theorem 3.1 from \cite{BKS}.

\begin{proposition}
Let $(X,d_X)$ be a metric space and $(Y,d_Y)$ be a separable metric space. Let  $(f_n)$ be a sequence of  Baire~1 functions from $X$ to $Y$. If $(f_n)$ converges uniformly on bounded sets to
$f\colon X \to Y$, then the family $\{f_n: n \in \mathbb{N}\}$ is equi-Lebesgue. Thus,
by Fact \ref{Ali}, it is equi-Baire~1.
\end{proposition}
\begin{proof}
Let $x_0$ be an arbitrary point in $X$. For every $n \in \mathbb{N}$ put
$B_n = B(x_0,n)$. Let $l \in \mathbb{N}$ and $\varepsilon > 0$.
Put $\mathcal F = \{f_1, f_2, \dots, f_n, \dots \}$. Then $\mathcal F$ is a totally bounded set in $(Y^X,\tau_{UB})$. There are functions $f_{k_1}, f_{k_2}, \dots f_{k_n} \in \mathcal F$ such that
	$$\mathcal F \subseteq \bigcup \{W(f_{k_j},B_l,\varepsilon/3): j\in\{1,\dots , n\}\}.$$
For simplicity  put $g_j = f_{k_j}$ for every $j\in\{1,\dots , n\}.$

Let $\mathcal V = \{V_1, \dots, V_n, \dots\}$ be a countable open cover of $Y$, where diameter of members of this cover is less than $\varepsilon/3$. For every  $j\in\{1,\dots, n\}$ and  $i \in \mathbb{N}$ put $B_i^j  = g_j^{-1}(V_i)$.
Since $g_j$ is a Baire 1 function, $B_i^j$ is an $F_\sigma$ set for all $j\in\{1,\dots, n\}$ and $i\in \mathbb{N}$. Put
$$\mathcal B' = \{B_l \cap B_{i_1}^1 \cap \dots \cap B_{i_n}^n: i_j \in \mathbb{N},\; j \in \{1,\dots, n\}\}.$$
Every set $B \in \mathcal B' $ is of type $F_\sigma$. Let $\mathcal{B}$ be the family containing all nonempty sets from $\mathcal B'$.  Clearly, $B_l =\bigcup\mathcal B$ and the family $\mathcal B$ is countable. One can check that for all $f_i \in \mathcal F$ and for all  $B\in\mathcal{B}$  we have $\on{diam}f_i(B)<\varepsilon$.
Since $X  =  \bigcup\{B_l: l \in \mathbb{N}\}$, we are done.
\end{proof}

\section{Further results}

The notion of a cliquish function was introduced by Thielman in \cite{Th}.

\begin{df} \cite{Th}
Let $X$ be a topological space and $(Y,d)$ be a metric space.
A function $f\colon X  \to Y$ is called cliquish at $x \in X$ if for every $\varepsilon > 0$ and for
each open neighbourhood $U$ of $x$ there is a nonempty open set $G \subset U$ such that
$d(f(x), f(y)) < \varepsilon$ for all $x, y  \in G$. The function $f\colon X \to Y$ is cliquish if it is cliquish at every point $x \in X$.
\end{df}

In \cite{Ho} the following result was proved. (We preserve our
definition of a Baire~1 function in a more general case.)

\begin{proposition} \label{h2}
Let $X$ be a Baire space and $(Y,d)$ be a separable metric space.
If $f\colon X \to Y$ be a Baire~1 function, then $f$ is cliquish.
\end{proposition}

From the proof of Proposition \ref{h2} given in \cite{Ho} we can derive the following stronger result.

\begin{proposition} \label{h22}
Let $(X,d_X)$ be a Baire metric space and $(Y,d_Y)$ be a  metric space.
If $f\colon X \to Y$ has the Lebesgue property, then $f$ is cliquish.
\end{proposition}

It is known (see \cite{E}) that, if $X$ is a Baire space and $(Y,d)$ is a metric space, then $f\colon X \to Y$ is cliquish if and only if the set of continuity points of $f$ is dense in $X$.
A function with this latter property is called pointwise discontinuous, cf. \cite[p. 105]{Ku}.

In \cite[Def. 4]{BHH}, the notion of equi-cliquishness of a sequence of functions from a topological space into a metric space was introduced.
We suggest the following definition for families of functions.

\begin{df}
Let $X$ be a topological space and $(Y,d)$ be a metric space. A family $\mathcal F$ of
functions from $X$ to $Y$ is called equi-cliquish at $x \in X$ if for every $\varepsilon > 0 $
and each neighbourhood $U  \subset X$  of $x$ there is a nonempty open set $G \subset U$  such that
$\on{diam} f(G)\le \varepsilon$ for all $f \in \mathcal F$.
A family $\mathcal F$ of functions from $X$ to $Y$ is equi-cliquish if it is equi-cliquish at every $x \in X$.
\end{df}

For a family $\mathcal F$ of functions from a topological space $X$ to a metric space $(X,d)$ we denote by $EC(\mathcal F)$ the set of equi-continuity points of $\mathcal F$. It is known that this set is
$G_\delta$ since $EC(\mathcal F)=\bigcap_{n \in \mathbb N} O(1/n)$ where
$$
O(\varepsilon) = \bigcup\{U: U \subset X,\; U \mbox{ nonempty open, }\on{diam}f(U) < \varepsilon \mbox{ for every } f \in \mathcal F\}
$$
is open for each $\varepsilon >0$; see \cite[Lemma 3.2]{FHT}.

Note that in \cite{FHT} the authors defined the notion of almost equicontinuity for a family of continuous functions, which is  similar to the notion of equi-cliquishness.

\begin{df} \cite{FHT} Let $X$ and $(M, d)$  be a Hausdorff, completely regular space and a metric space, respectively, and let
$C(X,M)$ denote the set of all continuous functions from $X$ to $M$. A subset $G \subset C(X,M)$ is said to be
almost equicontinuous if $G$ is equicontinuous on a dense subset of $X$.
\end{df}

\begin{proposition} \label{new}
Let $(X,d_X)$ be a Baire metric space and $(Y,d_Y)$  be a  metric space. If $\mathcal F$ is an equi-cliquish  family of functions from $X$ to $Y$, then the set $EC(\mathcal F)$ is dense $G_\delta$, and thus it is comeager.
%there is a dense $G_\delta$-set  $H \subset X$ such that $\mathcal F$ is equicontinuous at every point $x \in H$.
\end{proposition}
\begin{proof}
Since $\mathcal F$ is equi-cliquish, the set $O(\varepsilon)$ given above is dense in $X$. Since $X$ is a Baire space, the $G_\delta$ set $EC(\mathcal F) = \bigcap_{n \in \mathbb N} O(1/n)$ is dense and thus it is comeager.
%It is easy to verify that the family $\mathcal F$ is equicontinuous at every point $x \in H$.
\end{proof}

\begin{proposition} \label{h11}
Let $(X,d_X)$ be a Baire metric space and $(Y,d_Y)$  be a  metric space. If $\mathcal F$ is an equi-Lebesgue  family of functions from $X$ to $Y$, then $\mathcal F$ is equi-cliquish.
\end{proposition}

\begin{proof}
Let $x \in X$ and $U \subset X$ be an open neighbourhood of $x$ and $\varepsilon > 0$. Since  the family $\mathcal F$ is equi-Lebesgue, there is a cover $(X_i)_{i \in \mathbb N}$  of $X$ consisting of closed sets such that $\on{diam}f(X_i) \le \varepsilon$ for all $i \in \mathbb N$ and $f \in \mathcal F$. The set $U$ is $F_\sigma$. Let $\{F_n: n \in \mathbb N\}$ be a family of closed sets in $X$ such that $U = \bigcup_{n \in \mathbb N} F_n$. Then
$$U = \left(\bigcup_{n \in \mathbb N} F_n\right) \cap \left(\bigcup_{i \in \mathbb N} X_i\right) = \bigcup_{n \in \mathbb N} \bigcup_{i \in \mathbb N} (F_n \cap X_i).$$
Since $X$ is a Baire space, there are $n, i \in \mathbb N$ and nonempty open set $V \subset U$ such that $V \subset F_n \cap X_i$. Thus $\on{diam}f(V) \le \varepsilon$ for every $f \in \mathcal F$. Hence $\mathcal F$ is equi-cliquish at $x \in X$. Since $x$ was arbitrary, $\mathcal F$ is equi-cliquish.
\end{proof}

\begin{corollary} \label{hl30}
Let $(X,d_X)$ be a Baire metric space and $(Y,d_Y)$  be a  metric space. Then the set $EC(\mathcal F)$ is dense  $G_\delta$  for every family $\mathcal F$ of
equi-Lebesgue  family of functions from $X$ to $Y$. If additionally,
$(X,d_X)$ is separable, the set $EC(\mathcal F)$ is dense $G_\delta$ for every equi-Baire~1 family
of functions from $X$ to $Y$.
\end{corollary}
\begin{proof}
For the former assertion we use Propositions \ref{new} and \ref{h11}. The latter assertion follows
from Fact \ref{Ali} (ii).
\end{proof}

Of course, in general, we cannot expect that in Corollary \ref{hl30} an equi-Baire 1  family of functions will be equicontinuous on the whole space.

\begin{example}
Let $X = [0,1]$ and $Y = \mathbb R $ be equipped  with the usual Euclidean metric. Define $f\colon X \to Y$ by $f(x) = \sin(1/x)$ if $x \ne 0$ and $f(0) = 0$. The family $\{f\}$ is equi-Baire 1, however it cannot be equicontinuous at $0$. Example~3.8 in \cite{Al} presents an equi-Baire1 family of continuous functions from $X$ to $Y$ which is not equicontinuous at $0$.
\end{example}

Natural questions arise whether a family $\mathcal F$ of functions from a metric space $(X,d_X)$ to a metric space $(Y,d_Y)$, with the set $H=EC(\mathcal F)$ dense in $(X,d_X)$, is equi-Baire~1 or, whether the functions $f|H$ for $f\in\mathcal F$ can be extended to an equi-Baire~1 family on $X$.
Firstly, we show that the answer to the first question is negative.

\begin{example}
S. Marcus in his paper \cite{Ma} showed that there is a quasicontinuous function $f$ from
the interval $[0,1]$ to $\mathbb R$ which is not Lebesgue measurable. Thus $f$ cannot be a Baire~1 function. It is known that every quasicontinuous function from a Baire space into a metric space is continuous at points of a dense $G_\delta$ set (see \cite{HHM}).
\end{example}

We also have a positive answer to our second question in a full generality.

\begin{theorem} \label{Last}
Let $(X,d_X)$ be a separable metric space and $(Y,d_Y)$ be a separable complete metric space.
Let $H\subset X$ be a nonempty $G_\delta$ set, and $\mathcal F$ be an equi-continuous family of functions
from $H$ to $(Y,d_Y)$. Then all functions in $\mathcal F$ can be extended to an equi-Baire~1 family of functions from $X$ to $Y$.
\end{theorem}
\begin{proof}
We use ideas from \cite[pp. 434--435]{Ku} with $\alpha=1$. Fix $\varepsilon>0$. Then for each $x\in H$
there is an open set $O_x$ in $H$ such that $x\in O_x$ and $\on{diam} f(O_x)<\varepsilon/2$ for every $f\in\mathcal F$. Since $X$ has a countable base, we can suppose that there are open sets $G_1, G_2, \dots$ in $H$ such that $H=\bigcup_{n\in\mathbb N}G_n$ and $\on{diam} f(G_n)<\varepsilon/2$
for all $f\in\mathcal F$ and $n\in\mathbb N$. There is a countable discrete set $I\subset Y$
such that for every $y\in Y$ there is $z\in I$ such that $d_Y(y,z)<\varepsilon/2$.

For all $f\in\mathcal F$ and $i\in\mathbb N$, pick $y_i^f\in I\cap S(f(G_i),\varepsilon/2)$. Then for every
$f\in\mathcal F$ let $g^f\colon h\to Y$ be defined as follows: $g^f(x)=y_1^f$ if $x\in G_1$, and $g^f(x)=y_k^f$ if $x\in G_k\setminus\bigcup_{i<k}G_i$ and $k>1$. It is easy to check that
$g^f\colon H\to Y$ is Baire~1 and $d_Y(f(x),g^f(x))<\varepsilon$ for each $x\in H$.

Now, for every $f\in\mathcal F$ and each $n\in\mathbb N$, let $g_n^f\colon H\to Y$ be the function $g^f$ defined for $\varepsilon =1/2^{n+1}$. It is visible that $(g_n^f)_{n\in\mathbb N}$ converges uniformly to $f$ on $H$.
Note that for all $n,k\in\mathbb N$ and every $f\in\mathcal F$ we have
$$
d_Y(g_n^f(x),g_k^f(x))\le 1/2^{k}\;\;\mbox{for each } x\in H \mbox{ if }n\ge k.
$$
By the construction, the image $g_n^f(H)$ is a countable discrete set, for all $f\in\mathcal F$ and $n\in\mathbb N$.

Following the
proof in \cite{Ku}, we can extend every $g_n^f$ to a Baire~1 function
$h_n^f\colon X\to g_n^f(H)$ in such a way that
\begin{equation} \label{equa2}
d_Y(h_n^f(x),h_k^f(x))\le 1/2^{k}\;\;\mbox{for each } x\in H \mbox{ if }n\ge k.
\end{equation}
Moreover, given $n\in\mathbb N$, all the functions $h_n^f$ with $f\in\mathcal F$, are defined in the same way in the sense that $X=\bigcup_{i\in\mathbb N}F_i^n$ where $F_i^n$ is an $F_\sigma$ set in $X$  and
$|h_n^f(F_i^n)|=1$ for each $i\in\mathbb N$. Since $(Y, d_Y)$ is complete, condition (\ref{equa2})
implies that, for every $f\in\mathcal F$, the (uniform) limit $\lim_n h_n^f(x)=f^*(x)$ exists for all $x\in X$ which defines the extension $f^*\colon X\to Y$ of $f$. Also, if $n\to\infty$ in
(\ref{equa2}), we obtain $d_Y(f^*(x), h_k^f(x))\le 1/2^{k}$ for all $x\in X$.

We will prove that $\{ f^*\colon f\in\mathcal F\}$ is equi-Lebesgue. This yields the assertion by  Fact 3.1.(i) which claims that every equi-Lebesgue family is equi-Baire. It is enough to show that,
for each $k\in\mathbb N$, we have $\on{diam} f^*(F_i^k)\le 1/2^{k-1}$ for all $i\in\mathbb N$ and $f\in\mathcal F$. So, fix $k,i\in\mathbb N$ and $f\in\mathcal F$. For any $x,x'\in F_i^k$ we have
$$d_Y(f^*(x),f^*(x'))\le d_Y(f^*(x), h_k^f(x))+d_Y(h_k^f(x), h_k^f(x'))+d_Y(f^*(x'), h_k^f(x'))\le 1/2^{k-1}$$
where $d_Y(h_k^f(x), h_k^f(x'))=0$ since $|h_k^f(F_i^k)|=1$.
\end{proof}

\bibliographystyle{amsplain}

\end{document}